\documentclass[final]{dmtcs}
\usepackage[latin1]{inputenc}
\usepackage{subfigure}
\usepackage{amsmath,latexsym,amssymb,dsfont}
\usepackage[round]{natbib}

\RCSdef$Revision: 1.4 $\endRCSdef
\rcsMajMin

\newtheorem{Th}{Theorem}[section]
\newtheorem{Lemma}[Th]{Lemma}
\newtheorem{Cor}[Th]{Corollary}
\newtheorem{Prop}[Th]{Proposition}

\newcommand{\E}{\mathbb{E}}
\newcommand{\Prob}{\mathbb{P}}

\newcommand{\calZ}{\mathcal{Z}}
\newcommand{\N}{\mathbb{N}}
\newcommand{\Z}{\mathbb{Z}}

\numberwithin{equation}{section}

\author{Lorenz A. Gilch\addressmark{1}}
\title[Rate of Escape of Random Walks on Regular Languages and Free Products by
Amalgamation]{Rate of Escape of Random Walks on Regular Languages and Free
  Products by Amalgamation of Finite Groups}
\address{\addressmark{1}Institut f\"ur Mathematische Strukturtheorie (Math. C),
  Graz University of Technology, Steyrergasse
  30, A-8010 Graz, Austria}
\accepted{13 Jun 2008}
\revised{14 Jul 2008}

\begin{document}
\maketitle
\begin{abstract}
We consider random walks on the set of all words over a finite alphabet such
that in each step only the last two letters of the current word may be modified
and only one letter may be adjoined or deleted. We assume that the transition
probabilities depend only on the last two letters of the current
word. Furthermore, we consider also the special case of random walks on free
products by amalgamation of finite
groups which arise in a natural way from random walks on the
single factors. The aim of this paper is to compute several equivalent formulas
for the rate of escape with respect to natural length functions for these random walks using different techniques.
\end{abstract}

\section{Introduction}

Let $A$ be a finite alphabet and let $A^\ast$ be the set of all finite words
over the alphabet $A$, where $\varepsilon$ is the empty
word. Furthermore, let $l:A \to [0;\infty)$ be a function representing a
`letter length'. The extension of $l$ to $A^\ast$ defined by $l(a_1\dots
a_n)=\sum_{i=1}^n l(a_i)$ gives then a suitable 'word length'.
We consider a transient Markov chain $(X_n)_{n\in\mathbb{N}_0}$ on $A^\ast$
with $X_0=\varepsilon$ such that transition probabilities depend only on the
last two letters of the actual word and in each step only the last two letters
may be modified and only one letter may be
adjoined or deleted. We are interested in whether
the sequence of random variables $l(X_n)/n$ converges almost surely to a
constant, and if so, to compute this constant. If the limit exists, it is called the
\textit{rate of escape}, or the \textit{drift with respect to $l$}. In this
paper, we study this question for random walks on regular languages and on free
products by amalgamation of finite groups, which form
special cases of regular languages and are a generalization of free
products of groups.
\par
It is well-known that the rate of escape w.r.t. the natural word length exists
for a random walk on a finitely generated group, which is governed by a
probability measure on the group. This follows from \textit{Kingman's subadditive
  ergodic  theorem}; see Kingman \cite{kingman}, Derriennic \cite{derrienic} and
Guivarc'h \cite{guivarch}. 
There
are many detailed results for random walks on free products by
amalgamation:
Picardello and Woess \cite{picardello-woess} showed that a locally compact free
product by amalgamation of compact groups acts naturally on a tree. They also derived the
behaviour of the $n$-step transition probabilities. 
Cartwright and Soardi \cite{cartwright-soardi} investigated random walks on
free products by amalgamation, where the amalgamating subgroup is finite and normal. They derived a
formula for the Green function $G(z)=\sum_{n\geq 0} p^{(n)}(e,e)z^n$, where
$p^{(n)}(e,e)$ is the $n$-step return probability from the identity $e$, of the
random walk on the amalgamated product in terms of 
the Green functions of the single factors that is essentially the same as in
Woess \cite{woess3}. For random walks on
free products of finite groups Mairesse and Math\'eus \cite{mairesse1} have
developed a specific technique for the computation of the rate of
escape. For this purpose, they have to solve a more elegant system of algebraic equations than we have to solve, but
our results will be more general.
Three different formulas for the rate of escape of  random walks on free
products of graphs and groups are derived in Gilch \cite{gilch}. The techniques used in \cite{gilch} were the
starting point for the computation of the rate of escape in this paper.
 An important link
between drifts and harmonic analysis was obtained by Varopoulos
\cite{varopoulos}. He proved that for symmetric finite range random walks on groups the existence of non-trivial bounded
harmonic functions is equivalent to a non-zero rate of escape. The recent work
of Karlsson and Ledrappier \cite{karlsson-ledrappier07} generalizes this result to random walks with finite first moment of the step lengths. This leads to a
link between the rate of escape and the entropy of random walks; compare
e.g. with Kaimanovich and Vershik \cite{kaimanovich-vershik} and Erschler
\cite{erschler2}. 
\par
We also consider random walks on regular languages which can be seen as a generalization
of free products by amalgamation. Random Walks on this class of structures have been
investigated by several authors:
\mbox{Malyshev \cite{malyshev},} \cite{malyshev-inria} and Gairat, Malyshev, Menshikov, Pelikh  \cite{malyshev2}
stated criteria for transience, null-recurrence and positive
recurrence. More\-over, \mbox{Malyshev} proved limit theorems concerning existence
of the stationary distribution and speed in the transient case and convergence
of conditional distributions in the ergodic case; in particular, he
showed that the rate of escape w.r.t. the natural word length (that is, $l(\cdot)=1$) is constant and it is strictly
positive if and only if the random walk is transient. Yambartsev and Zamyatin
\cite{yambartsev-zamyatin} proved limit theorems for random walks on two
semi-infinite strings over a finite alphabet.
Lalley \cite{lalley} also investigated random walks on regular languages. He found out that the $n$-step return probabilities must obey one of three
different types of power laws. His analysis is based on a finite algebraic
system of generating functions related to the Green \mbox{function.} This
algebraic system is also used in this paper to compute explicit formulas for
the rate of escape. 
The rate of escape has also been studied on trees, which may be seen as a
special case of our context:
 Nagnibeda and Woess \cite[Section 5]{woess2} proved that the
rate of escape of random walks on trees with finitely many cone types is
non-zero and give a formula for it. One of the techniques used in this paper
for the computation of the rate of escape was motivated by Nagnibeda and Woess.
\par
Our aim is to compute formulas for the rate of escape of random walks on
regular languages and free products by amalgamation of finite groups. In Section \ref{reglanguages}
we compute the rate of escape of random walks on regular languages, while in
Section \ref{amalgams} we compute it for random walks on free products by amalgamation. In Section
\ref{exit-time} we compute the rate of escape analogously to Section
\ref{rate-of-escape} and in Section \ref{dgf} we compute it by an application of a theorem
of Sawyer and Steger \cite{sawyer}. In Section \ref{limitprocesses} we use the
algebraic group structure of free products by amalgamation to compute the rate of escape with
respect to the natural word length. This
approach is based on a technique
which was already used by Ledrappier \cite{ledrappier} and Furstenberg
\cite{furstenberg}. Finally,
in Section \ref{sample} we give sample computations.

\section{Rate of Escape of Random Walks on Regular Languages}
\label{reglanguages}

\subsection{Regular Languages and Random Walks}

Let $A$ be a finite alphabet and $\varepsilon$ be the empty word. A \textit{random walk on a regular language} is a Markov
chain on the set $A^\ast:=\bigcup_{n\geq 1} A^n\cup \{\varepsilon\}$ of all
finite words over the alphabet $A$, whose transition probabilities obey the
following rules:
\begin{enumerate}
\item[(i)] Only the last two letters of the current word may be modified.
\item[(ii)] Only one letter may be adjoined or deleted at one instant of time.
\item[(iii)] Adjunction and deletion may only be done at the end of the current word.
\item[(iv)] Probabilities of modification, adjunction or deletion depend only on the
  last two letters of the current word.
\end{enumerate}
Compare with Lalley \cite{lalley}. The hypothesis that transition probabilities
depend only on the last two letters of the current word can be weakened to
dependence of the last $K\geq 2$ letters by a ``recoding trick'', which is also
described by Lalley. In general, a regular language is a subset of $A^\ast$
whose words are accepted by a finite-state automaton. It is necessary that by
each modification of a word of the regular language in one single step a new
word of the regular language is created. The results below, however, are so
general such that w.l.o.g. \mbox{-- for ease } and better readability -- we may
assume that the regular language consists of the whole \mbox{set $A^\ast$.}
\par
The random walk on $A^\ast$ is described by the sequence of random variables
$(X_n)_{n\in\N_0}$. Initially, we have $X_0=\varepsilon$.
For two words $w,w'\in A^\ast$ we write $ww'$ for the concatenated word.
We use the following abbreviations for the transition probabilities: for $w\in
A^\ast$, $a,a',b\in A$, $b',c'\in A\cup \{\varepsilon\}$, $n\in\N_0$, let be
$$
\begin{array}{c}
\Prob[X_{n+1}=wa'b'c'\mid X_n=wab]  =  p(ab,a'b'c'),\\[0.5ex]
\Prob[X_{n+1}=b'c'\mid X_n=a]  =  p(a,b'c'),\\[0.5ex]
\Prob[X_{n+1}=b'\mid X_n=\varepsilon] =  p(\varepsilon,b').
\end{array}
$$
If we want to start the random walk at $w\in A^\ast$ instead of $\varepsilon$,
we write for short $\Prob_w[\,\cdot\,]:=\Prob[\,\cdot \mid X_0=w]$. 
Suppose
we are given a function $l:A\to [0;\infty)$. We extend $l$ to $A^\ast$ by defining
$l(a_1a_2\dots a_n):=\sum_{i=1}^n l(a_i)$ for \mbox{$a_1a_2\dots a_n\in A^n$.} Additionally, we set $l(\varepsilon):=0$. If $l(a)=1$ for each $a\in A$,
then $l$ is just the \textit{natural word length} which is denoted by $|\cdot
|$. If there is a non-negative constant $\ell$ such that
$$
\lim_{n\to\infty}\frac{l(X_n)}{n} = \ell \quad \textrm{ almost surely,}
$$
then $\ell$ is called the \textit{rate of escape} with respect to $l$. 
Malyshev \cite{malyshev} proved that the rate of escape w.r.t. the natural word length
exists. Furthermore, by Malyshev follows that the rate of escape w.r.t. $l$ is zero if and only if $(X_n)_{n\in\N_0}$ is recurrent. Our aim is
to compute a formula for $\ell$ in the transient case. Therefore, we assume
from now on transience of $(X_n)_{n\in\N_0}$.
\par
Moreover, we assume that the random walk on $A^\ast$ is \textit{suffix-irreducible},
that is, for all $w\in A^\ast$ with $\Prob[X_m=w]>0$ for some $m\in\N$ and for
all $ab\in A^2$ there is some $n\in\N$ such that 
$$
\Prob\Bigl[ \exists w_1\in A^\ast : X_n=ww_1ab, \forall k<n: |X_k|\geq |w| \,\Bigl|\, X_0=w\Bigr]>0.
$$
If suffix-irreducibility is dropped, then the rate of escape may be
non-deterministic; e.g., if $A=\{a,b\}$ with $l(a)=l(b)=1$ and
$p(aa,aaa)=p>1/2$, $p(aa,a)=p(a,\varepsilon)=1-p$,
$p(\varepsilon,a)=p(\varepsilon,b)=1/2$, 
$p(bb,bbb)=q>1/2$, $p(bb,b)=p(b,\varepsilon)=1-q$ with $p\neq
q$, then $l(X_n)/n$ converges only non-deterministically.

\subsection{The Rate of Escape}
\label{rate-of-escape}
The technique we use to compute $\ell$ was motivated by Nagnibeda and Woess
\cite[Section 5]{woess2}. For $k\in\N_0$ we define the \textit{$k$-th exit time} as
$$
\mathbf{e}_k:=\sup \bigl\lbrace m\in\N_0 \,\bigl|\, |X_m|=k\bigr\rbrace.
$$
As the alphabet $A$ is finite
and the random walk on $A^\ast$ is assumed to be transient, we have
$\mathbf{e}_k<\infty$ almost surely for every $k\in\N_0$. Furthermore, we write $\mathbf{W}_k:=X_{\mathbf{e}_k}$ and
$\mathbf{i}_k:=\mathbf{e}_k-\mathbf{e}_{k-1}$ with $\mathbf{e}_{-1}:=0$. We
show at first that $(\mathbf{W}_k,\mathbf{i}_k)_{k\geq 3}$ is a Markov chain. For this purpose, we
introduce some useful functions: for $a,b,c\in A$ and real $z>0$ define
\begin{eqnarray*}
H(ab,c|z) & := & \sum_{n=1}^\infty \Prob_{ab}\bigl[X_n=c,\forall m<n:
|X_m|>1\bigr]\, z^n,\\
\xi(abc) & :=& \sum_{a'b'c'\in A^3} p(bc,a'b'c')\cdot \Bigl(1- \sum_{d\in A} H(b'c',d|1)\Bigr).
\end{eqnarray*}
Observe that
$$
\Prob_{abc}\bigl[ X_n=ab', \forall m<n:|X_m|>2\bigr]
= \Prob_{bc}\bigl[ X_n=b',\forall m<n:|X_m|>1\bigr],
$$
as the transition probabilities depend only on the last two letters of the
current word and in each step only one letter may be deleted. Thus, the number $\xi(abc)$ is the probability of starting at $abc\in A^3$ such that
$|X_n|\geq 4$ for all $n\geq 1$, and it does not depend on the letter
``$a$''. Furthermore, let be $[a_1\dots a_n]_3:=a_{n-2}a_{n-1}a_n$, if
$a_1\dots a_n\in A^\ast$ with $n\geq 3$.
With this notation we get:
\begin{Prop}\label{prop-e_k}
The stochastic process $(\mathbf{W}_k,\mathbf{i}_k)_{k\geq 3}$ is a Markov chain with transition probabilities
\begin{eqnarray*}
&& \mathbb{P}\bigl[\mathbf{W}_{k+1}=x_{k+1},\mathbf{i}_{k+1}=n_{k+1}\, \bigl|\,
\mathbf{W}_k=x_k,\mathbf{i}_k=n_k\bigr]  \\[2ex]
& = & \frac{\xi([x_{k+1}]_3)}{\xi([x_{k}]_3)}\cdot \Prob_{x_k}\bigl[X_{n_{k+1}}=x_{k+1}, \forall i\in\{1,\dots,n_{k+1}\}: |X_i|>k\bigr]
\end{eqnarray*}
for $n_k,n_{k+1}\in\mathbb{N}$, $x_k,x_{k+1}\in A^\ast$ with $|x_k|=k$, 
$|x_{k+1}|=k+1$ and \mbox{$\Prob\bigl[\mathbf{W}_k=x_k,\mathbf{i}_k=n_k\bigr]>0$.}
\end{Prop}
\begin{proof}
Let be $n_0,n_1,\dots,n_{k+1}\in\N$ and $x_0,x_1,\dots,x_{k+1}\in A^\ast$ with
$|x_j|=j$ for \mbox{$j\in\{0,1,\dots,k+1\}$.} Define the event
$$
\bigl[\mathbf{W}_0^m=x_0^m,\mathbf{i}_0^m=n_0^m\bigr] :=
\bigl[\forall j\in\{0,1,\dots,m\}:\mathbf{W}_j=x_j,\mathbf{i}_j=n_j\bigr],
$$
where $m\in\{k,k+1\}$. With this notation we get
\begin{eqnarray*}
\mathbb{P}\bigl[\mathbf{W}_0^k=x_0^k,\mathbf{i}_0^k=n_0^k\bigr]
&=& \Prob\left[
\begin{array}{c}
\forall j\in\{0,\dots,k\} \, \forall \lambda\in\{0,\dots,n_j\}:\\
|X_{n_1+\dots+n_{j-1}+\lambda}|\geq j, X_{n_1+\dots+n_{j}}= x_j
\end{array}
\right]
\cdot \Prob_{x_k}\bigl[\forall n\geq 1: |X_n|>k\bigr]\\
&=& \Prob\left[
\begin{array}{c}
\forall j\in\{0,\dots,k\} \, \forall \lambda\in\{0,\dots,n_j\}:\\
|X_{n_1+\dots+n_{j-1}+\lambda}|\geq j, X_{n_1+\dots+n_{j}}= x_j
\end{array}
\right]
\cdot \xi([x_k]_3).
\end{eqnarray*}
Analogously,
\begin{eqnarray*}
&&\mathbb{P}\bigl[\mathbf{W}_0^{k+1}=x_0^{k+1},\mathbf{i}_0^{k+1}=n_0^{k+1}\bigr]\\
&=& \Prob\left[
\begin{array}{c}
\forall j\in\{0,\dots,k\} \, \forall \lambda\in\{0,\dots,n_j\}:\\
|X_{n_1+\dots+n_{j-1}+\lambda}|\geq j, X_{n_1+\dots+n_{j}}= x_j
\end{array}
\right]\\[1ex]
&&\ 
\cdot \Prob_{x_k}\bigl[\forall i\in\{1,\dots,n_{k+1}\}:|X_i|>k,X_{n_{k+1}}=x_{k+1}\bigr]
\cdot \xi([x_{k+1}]_3).
\end{eqnarray*}
Thus, under the assumption that
$\mathbb{P}\bigl[\mathbf{W}_0^k=x_0^k,\mathbf{i}_0^k=n_0^k\bigr]>0$ we obtain
\begin{eqnarray*}
&&\mathbb{P}\bigl[\mathbf{W}_0^{k+1}=x_0^{k+1},\mathbf{i}_0^{k+1}=n_0^{k+1}\, \bigl|\,
\mathbf{W}_0^k=x_0^k,\mathbf{i}_0^k=n_0^k\bigr] \\
&=& 
\frac{\xi([x_{k+1}]_3)}{\xi([x_{k}]_3)}\cdot 
\Prob_{x_k}\bigl[\forall i\in\{1,\dots,n_{k+1}\}:|X_i|>k,X_{n_{k+1}}=x_{k+1}\bigr].
\end{eqnarray*}
\end{proof}

Observe that $\Prob_{x_k}\bigl[\forall
i\in\{1,\dots,n_{k+1}\}:|X_i|>k,X_{n_{k+1}}=x_{k+1}\bigr]$ depends only on
$n_{k+1}$, $[x_k]_3$ and $[x_{k+1}]_3$. We use this observation to construct a new
Markov chain on the state space
$$
\mathcal{Z} := \bigl\lbrace
(abc,n)\in \overline{A}^3\times \N \, \bigl| \, \exists de\in A^2:
\Prob_{de}[X_n=abc, \forall m\in\{1,\dots,n\}:|X_m|>2]
\bigr\rbrace,
$$
where $\overline{A}^3 :=\{abc \in A^3 \mid \xi(abc)>0\}$
with the following transition probabilities:
$$
q\bigl( (abc,n),(a'b'c',n')\bigr) =
\frac{\xi(a'b'c')}{\xi(abc)}\cdot
\Prob_{abc}\bigl[ X_n=aa'b'c', \forall i\in\{1,\dots,n'\}: |X_i|\geq 4\bigr].
$$
Observe that 
$$
\mathbb{P}\bigl[\mathbf{W}_{k+1}=x_{k+1},\mathbf{i}_{k+1}=n_{k+1}\, \bigl|\,
\mathbf{W}_k=x_k,\mathbf{i}_k=n_k\bigr] = q\bigl(([x_k]_3,n_k),([x_{k+1}]_3,n_{k+1})\bigr)
$$
for $k\geq 3$ and that the transition probabilities do not depend on
$n_k$. This provides that also $\bigl( [\mathbf{W}_k]_3\bigr)_{k\geq 3}$ is a
Markov chain on $\overline{A}^3$ with transition probabilities
$$
\tilde{q}(abc,a'b'c') =
\sum_{n'\in\N} q\bigl( (abc,n_{abc}),(a'b'c',n')\bigr), 
$$
where the $n_{abc}$'s on the right hand side of the equation may be chosen
arbitrarily. Observe that $[\mathbf{W}_k]_3$ may only take a finite number of states,
since the alphabet $A$ is finite and $|[\mathbf{W}_k]_3| = 3$. At this point we
need the above made assumption of suffix-irreducibility; this provides that
$\bigl([\mathbf{W}_k]_3\bigr)_{k\geq 3}$ is irreducible and therefore has an invariant probability measure
$\nu$. 
\begin{Lemma}
Let be $abc\in A^3$ and $n\in\N$ and define
$$
\pi(abc,n):=\sum_{def \in \overline{A}^3} \nu(def)\, q\bigl((def,n_{def}),(abc,n)\bigr),
$$
where $n_{def}$ can be chosen arbitrarily. Then $\pi$ is the unique invariant probability
measure of $\bigl([\mathbf{W}_k]_3,\mathbf{i}_k\bigr)_{k\geq 3}$.
\end{Lemma}
\begin{proof}
It is a straightforward computation to prove the lemma:
\begin{eqnarray*}
&& \sum_{(ghi,s)\in\mathcal{Z}} \pi(ghi,s)\ q\bigl((ghi,s),(abc,n)\bigr)\\
&=& \sum_{(ghi,s)\in\mathcal{Z}} \sum_{def\in \overline{A}^3} \nu(def)\
q\bigl((def,n_{def}),(ghi,s)\bigr) \ q\bigl((ghi,s),(abc,n)\bigr) \\
&=& \sum_{ghi\in \overline{A}^3} q\bigl((ghi,n_{ghi}),(abc,n)\bigr) \sum_{def\in \overline{A}^3}
\nu(def) \sum_{s\in\N} q\bigl((def,n_{def}),(ghi,s)\bigr)\\
&=& \sum_{ghi\in \overline{A}^3} q\bigl((ghi,n_{ghi}),(abc,n)\bigr)\ \nu(ghi) 
= \pi(abc,n).
\end{eqnarray*}
\end{proof}

Define $g:\mathcal{Z}\to\N: (abc,n) \mapsto n$. An application of the
\textit{ergodic theorem for positive recurrent Markov chains} yields
\begin{equation}\label{integral}
\frac{1}{k} \sum_{i=3}^k g\bigl([\mathbf{W}_k]_3,\mathbf{i}_k\bigr) = \frac{\mathbf{e}_k-
  \mathbf{e}_2}{k}\ \xrightarrow{k\to\infty}\ \int g(abc,n)\,d\pi \quad
\textrm{almost surely},
\end{equation}
if the integral exists. Our next aim is to ensure finiteness of this integal
and to compute a formula for it. For this purpose, we define
\begin{eqnarray*}
\overline{G}(ab,cd|z) &:=& \sum_{n=0}^\infty \Prob_{ab}\bigl[ X_n=cd,\forall
m\leq n:|X_m|\geq 2\bigr]\, z^n,\\
\mathcal{K}(ab,cde|z) & :=& \sum_{n=1}^\infty \Prob_{ab}\bigl[ X_n=cde, \forall
m \in\{1,\dots,n\}:|X_m|\geq 3\bigr]\, z^n\\
&=& \sum_{fg\in A^2} p(ab,cfg)\cdot z \cdot \overline{G}(fg,de|z),
\end{eqnarray*}
where $a,b,c,d,e\in A$ and $z>0$. 
We have the following linear system of equations:
\begin{eqnarray}\label{g-system}
\overline{G}(ab,cd|z) &=& \delta_{ab}(cd) + \sum_{c'd'\in A^2}
p(ab,c'd') \cdot z\cdot \overline{G}(c'd',cd|z) +\nonumber\\
&& \ + \sum_{c'd'e'\in A^3} p(ab,c'd'e')\cdot
z\cdot \sum_{f'\in A} H(d'e',f'|z)\cdot \overline{G}(c'f',cd|z).
\end{eqnarray}
Moreover, we also have the following finite system of equations:
\begin{eqnarray}\label{h-equations}
H(ab,c|z) &=& p(ab,c)\cdot z +  \sum_{de\in
  A^2} p(ab,de)\cdot z\cdot H(de,c|z)\nonumber\\
&& \ + \sum_{def\in A^3} p(ab,def)\cdot z \cdot \sum_{g\in A} H(ef,g|z)\cdot H(dg,c|z);
\end{eqnarray}
compare with Lalley \cite{lalley}. 
The system (\ref{h-equations}) consists of equations of quadratic order, and
thus the functions $H(\cdot,\cdot|z)$ are algebraic, if the transition
probabilities are algebraic. If one
has solved this system, then the linear system of equations (\ref{g-system})
can be solved easily. In particular, the functions
$\overline{G}(\cdot,\cdot|z)$ are also algebraic for algebraic transition probabilities. Observe that we can write
$$
\tilde q(abc,a'b'c') = \frac{\xi(a'b'c')}{\xi(abc)}\ \mathcal{K}(bc,a'b'c'|1),
$$
providing $\nu$ can be computed if (\ref{h-equations}) can be solved. Turning back to our integral in (\ref{integral}) we can
now compute:
\begin{Prop}\label{exit-th}
We have $\lim_{k\to\infty} \mathbf{e}_k/k=\Lambda$ almost surely, where
\begin{eqnarray*}
\Lambda:= 
\sum_{abc,def\in \overline{A}^3} \nu(def)\cdot \frac{\xi(abc)}{\xi(def)}\cdot 
\frac{\partial}{\partial z}\biggl[ \sum_{gh\in A^2} p(ef,agh)\cdot z \cdot \overline{G}(gh,bc|z)\biggr]\Biggl|_{z=1}.
\end{eqnarray*}
\end{Prop}
\begin{proof}
We compute straight-forward:
\begin{eqnarray*}
&& \int g(abc,n)\,d\pi\\
&=& \sum_{(abc,n)\in\mathcal{Z}} n\cdot \sum_{def\in \overline{A}^3} \nu(def)\cdot q\bigl(
(def,n_{def}),(abc,n)\bigr)\\
&=& \sum_{def\in \overline{A}^3} \nu(def) \sum_{(abc,n)\in\mathcal{Z}} n\cdot
\frac{\xi(abc)}{\xi(def)}\cdot \Prob_{def}\bigl[ X_n=dabc, \forall
m\in\{1,\dots,n\}:|X_m|\geq 4\bigr]\\
&=& \sum_{abc,def\in \overline{A}^3} \nu(def)\cdot \frac{\xi(abc)}{\xi(def)}\cdot 
\sum_{n\in\N} n\cdot \Prob_{def}\bigl[X_n=dabc, \forall
m\in\{1,\dots,n\}:|X_m|\geq 4\bigr]\\ 
&=& \sum_{abc,def\in \overline{A}^3} \nu(def)\cdot \frac{\xi(abc)}{\xi(def)}\cdot 
\frac{\partial}{\partial z}\bigl[\mathcal{K}(ef,abc|z) \bigr]\Bigl|_{z=1}.
\end{eqnarray*}
Finiteness of the integal is ensured if all functions $H(\cdot,\cdot |z)$ and
$\overline{G}(\cdot,\cdot|z)$ have radii of convergence bigger than $1$. But
this follows from Lalley \cite{lalley}: he proved that the Green functions of
random walks on regular languages have radii of convergence bigger \mbox{than $1$.}
\end{proof}

Now we can state an explicit formula for the rate of escape:
\begin{Th}
\label{reglang-roe}
There is some non-negative constant $\ell$ such that
$$
\ell = \lim_{n\to\infty} \frac{l(X_n)}{n} = \frac{\Delta}{\Lambda}>0 \quad
\textrm{almost surely,}
$$
where
$$
\Delta := \sum_{abc,def\in \overline{A}^3} \nu(def)\,l(a)\,\frac{\xi(abc)}{\xi(def)} \mathcal{K}(ef,abc|1).
$$
In particular, $\lim_{n\to\infty} |X_n|/n = 1/\Lambda$ almost surely.
\end{Th}
\begin{proof}
With $h:\mathcal{Z}\to\N$ defined by $h(abc,n):=l(a)$ we obtain
$$
\frac{1}{n}\sum_{k=3}^n h\bigl([\mathbf{W}_k]_3,\mathbf{i}_k\bigr)
\xrightarrow{n\to\infty} \int h\,d\pi = \lim_{m\to\infty} \frac{l(X_{\mathbf{e}_m})}{m}.
$$
Simple computations lead to the following formula for this limit:
$$
\Delta := \int h\,d\pi = \sum_{abc,def\in \overline{A}^3} \nu(def)\cdot l(a)\cdot \frac{\xi(abc)}{\xi(def)}\cdot \mathcal{K}(ef,abc|1).
$$
Defining \mbox{$\mathbf{k}(n):=\max\{k\in\N_0 \mid \mathbf{e}_k\leq n\}$} we
obtain analogously to Nagnibeda and Woess \cite[Proof of Theorem D]{woess2}
$$
\ell =  \lim_{n\to\infty} \frac{l(X_n)}{n} = \lim_{n\to\infty} \frac{l(X_{\mathbf{e}_{\mathbf{k}(n)}})}{\mathbf{k}(n)}\frac{\mathbf{k}(n)}{\mathbf{e}_{\mathbf{k}(n)}} 
= \frac{\Delta}{\Lambda}>0.
$$
\end{proof}

Observe that for algebraic transition probabilities the rate of escape is
obtained by solving the algebraic system of equations (\ref{h-equations}). This yields that the rate
of escape is also algebraic, if the transition probabilities are algebraic and
$l(\cdot)$ takes only algebraic values.

\section{Rate of Escape of Random Walks on Free Products by Amalgamation}
\label{amalgams}
In this section we compute three formulas for the rate of escape of random
walks on free products by amalgamation of finite groups. This class of
structures form special cases of regular languages.

\subsection{Free Products by Amalgamation}
Let be $2\leq r\in\N$. Consider finite groups
$\Gamma_1,\dots,\Gamma_r$ with identities $e_1,\dots,e_r$ and subgroups $H_1\subset
\Gamma_1$, $\dots,H_r\subset \Gamma_r$. We assume that $H_1,\dots,H_r$ are
isomorphic, that is, there is a finite group $H$ such that
there are isomorphisms
$\varphi_1:H\to H_1,\dots,\varphi_r:H\to H_r$. Thus, we identify in the
following each $H_i$ with $H$. To explain the concept of free products
by amalgamation, we give at first a simple example: consider $\Gamma_1=\Gamma_2=\Z/d\Z$, $d\in\N$ even,
and the subgroup $H=\Z/2\Z$. Let $\Gamma_1$ be generated by an element $a$, and
$\Gamma_2$ by an element $b$. 
The free product by amalgamation  $\Z/d\Z \ast_{\Z/2\Z}
\Z/d\Z$ consists then of all finite words over the alphabet
$\{a,b\}$, where we have the relations $a^{d/2}=b^{d/2}$ and
$a^d=b^d=\varepsilon$. That is, any two words which can be deduced from each
other with these relations represent the same element. The relation $a^{d/2}=b^{d/2}$ means that the subgroup $\Z/2\Z$ in both copies of
$\Z/d\Z$ are identified. E.g., for $d=4$ it is $a^3bab^2=ab^3a^3=aba$.
To help visualize the concept of free products
by amalgamation, we may also think of the Cayley graphs $X_i$ of $\Gamma_i$. We
connect the graphs $X_i$ by identifying the subgroups $H=H_i$; at each
non-trivial coset of $H$ in all graphs $X_i$ we attach copies of $X_j$, $j\neq i$,
where the coset is identified with $H$ of the copy of $X_j$. This construction
is then iterated.
\par
We explain below free products by amalgamation in more detail. The quotient $\Gamma_i / H$ consists of all sets of sets $yH=\{yh \mid h\in H\}$, where $y\in\Gamma_i$. We fix
representatives $x_{i,1}=e_i,x_{i,2},\dots,x_{i,n_i}$ for the elements of
$\Gamma_i / H$, that is, for each $y\in\Gamma_i$ there is a unique
$x_{i,k}$ with $y\in x_{i,k}H$. We write $\Gamma_i^\times :=\Gamma_i\setminus H$ and $R_i:=\{x_{i,2},\dots,x_{i,n_i}\}$
with $n_i=[\Gamma_i:H]$. For any element $x\in\bigcup_{i=1}^r \Gamma_i$ we set
$\tau(x):=i$, if $x\in\Gamma_i^\times$, and $\tau(x):=0$, if $x\in H$.
\par
The free product of $\Gamma_1,\dots,\Gamma_r$
by amalgamation with respect to $H$ is given by
$$
\Gamma := \Gamma_1\ast_H \Gamma_2 \ast_H \dots \ast_H \Gamma_r,
$$
which consists of all finite words of the form
\begin{eqnarray}\label{word-form}
x_1x_2\dots x_n h,
\end{eqnarray}
where $h\in H$, $n\in\mathbb{N}_0$ and $x_1,\dots,x_n\in \bigcup_{i=1}^r R_i$ such that $\tau(x_i)\neq
\tau(x_{i+1})$. In the following we will always use this representation of
words. Suppose we are given a function \mbox{$l:\bigcup_{i=1}^r R_i\to
  [0;\infty)$.} Then we extend $l$ to a length function on $\Gamma$ by setting
$l(x_1\dots x_nh):=\sum_{i=1}^r l(x_i)$. 
The \textit{natural word length} is defined to be $\Vert x_1\dots x_nh
\Vert:=n$. In particular, $l(h)=\Vert h\Vert=0$ for all $h\in H$. 
For two words \mbox{$w_1=x_1x_2\dots x_m h$,} $w_2=y_1y_2\dots y_nh' \in\Gamma$
a group operation is defined in the following way: first, concatenate the two
words, then replace $hy_1$ in the middle by $y_1'h_1$ such that $y_1'$ is a
representative for the class of $hy_1$. Iterate the last step with $h_1y_2$ and
so on. Finally, we get a word of the form $x_1\dots x_ny_1'\dots y_n'h_n$ with
$h_n\in H$, that is, we get the requested equivalent
form (\ref{word-form}) for the concatenated word $w_1\circ w_2$. Note also that
$w^{-1}=h^{-1}x_m^{-1}\dots x_1^{-1}$ is the inverse of $w_1$ and can be
written in the form of (\ref{word-form}). The empty word $e$ is the identity of
this group operation. Observe that each $\Gamma_i$ is a subset of $\Gamma$.
\par
Suppose we are given probability measures $\mu_i$ on
$\Gamma_i$. Let $\alpha_1,\dots,\alpha_r$ be strictly positive real numbers such that
$\sum_{i=1}^r \alpha_i=1$. A probability measure on $\Gamma$ is given by 
$$
\mu(x) :=
\begin{cases}
\alpha_{\tau(x)} \mu_{\tau(x)}(x), & \textrm{if } x\in
\bigcup_{i=1}^r\Gamma_i^\times\\
\sum_{i=1}^r \alpha_i \mu_i(x), & \textrm{if } x\in H\\
0, & \textrm{otherwise}
\end{cases}.
$$
The $n$-th convolution power of $\mu$ is denoted by $\mu^{(n)}$.
The random walk $(X_n)_{n\in\N_0}$ on $\Gamma$ is then governed by the transition
probabilities $p(w_1,w_2):=\mu(w_1^{-1}w_2)$, where $w_1,w_2\in\Gamma$. Initially, $X_0:=e$.

\begin{Lemma}
The random walk on $\Gamma$ is recurrent if and only if $r=2=[\Gamma_1:H]=[\Gamma_2:H]$.
\end{Lemma}
\begin{proof}
Assume $r=2=[\Gamma_1:H]=[\Gamma_2:H]$. This provides $H\unlhd
\Gamma_1,\Gamma_2$, that is, $(\Gamma_1 \ast \Gamma_2)/ H \simeq (\Gamma_1 /
H)\ast (\Gamma_2/H)$  and $\Gamma_1 / H \simeq \mathbb{Z}/ 2\mathbb{Z} \simeq 
\Gamma_2 / H$. Since it is well-known that each random walk on the free product $(\mathbb{Z}/2\mathbb{Z}) \ast   (\mathbb{Z}/ 2\mathbb{Z})$, which arises from
a convex combination of probability measures on the single factors, is
recurrent, the random walk on $\Gamma$ also must be recurrent.
\par
Assume now that $r=2=[\Gamma_1:H]=[\Gamma_2:H]$ does not hold. Then either
$r\geq 3$ or w.l.o.g. $[\Gamma_1:H]\geq 3$. In both cases, $\Gamma$ is
non-amenable (for further details see e.g. Woess \cite[Th.10.10]{woess}). With Woess
\cite[Cor.12.5]{woess} we get that the random walk on $\Gamma$ must be transient.
\end{proof}

From now on we exclude the case $r=2=[\Gamma_1:H_1]=[\Gamma_2:H_2]$. 
In the following three subsections we want to compute three explicit formulas for the
rate of escape of our random walk on $\Gamma$. The first approach uses the
technique from the previous section, while the second approach arises from an
application of a theorem of Sawyer and Steger \cite{sawyer}. The third
technique uses the group
structure of $\Gamma$, but is restricted to the computation of the rate of
escape w.r.t. the natural word length.

\subsection{Exit Time Technique}
\label{exit-time}

We use the technique developped in Section \ref{rate-of-escape} to compute
$\ell$. Notice that $\Gamma$ is a special case of a regular language and our
random walk on $\Gamma$ fulfills the assumptions of our investigated random
walks on regular languages: starting from a word $x_1\dots x_nh\in \Gamma$ we
can only move in one step with positive probability to a word of the form
\begin{itemize}
\item $x_1\dots x_{n-1}x_n'h'$ with $x_n'h'\in \Gamma_{\tau(x_n)}^\times$,
  namely with probability $\mu(h^{-1}x_n^{-1}x_n'h')$, or
\item $x_1\dots x_{n}x_{n+1}h'$ with $x_{n+1}h'\in \bigcup_{i=1,i\neq
    \tau(x_n)}^r\Gamma_i^\times$, namely with probability \mbox{$\mu(h^{-1}x_{n+1}h')$,} or
\item $x_1\dots x_{n-1}h'$ with $h'\in H$, namely with probability $\mu(h^{-1}x_n^{-1}h')$,
\end{itemize}
where $x_1,\dots,x_{n+1},x_n'\in \bigcup_{i=1}^r R_i$ and $h,h'\in H$.
\par
We may now apply the technique of Section \ref{rate-of-escape} with some
slight modifications and simplifications. The exit-times are now given by
$$
\mathbf{e}_k:=\sup \bigl\lbrace m\in \N_0 \,\bigl|\, \Vert X_m\Vert=k \bigr\rbrace.
$$
Analogously, $\mathbf{W}_k:=X_{\mathbf{e}_k}$ and
$\mathbf{i}_k:=\mathbf{e}_k-\mathbf{e}_{k-1}$. We define for any $x,y\in R_i$, $i\in\{1,\dots,r\}$, $h,h'\in H$,
\begin{eqnarray*}
H(xh,h'|z) & := & \sum_{n=1}^\infty \Prob_{xh}\bigl[X_n=h',\forall m<n:
\Vert X_m\Vert \geq 1\bigr]\, z^n,\\
\xi(i) & :=& \sum_{gh_1\in \bigcup_{j=1,j\neq i}^r \Gamma_j^\times} \mu(gh_1)\cdot
\Bigl(1- \sum_{h_2\in H} H(gh_1,h_2|1)\Bigr)>0,\\
\overline{G}(xh,yh'|z) &:=& \sum_{n=0}^\infty \Prob_{xh}\bigl[ X_n=yh',\forall
m \leq n:\Vert X_m\Vert\geq 1\bigr]\, z^n.
\end{eqnarray*}
The functions $H(xh,h'|z)$ and $\overline{G}(xh,yh'|z)$ can be computed by
solving a finite system of non-linear equations; compare with (\ref{g-system}) and (\ref{h-equations}).
Analogously to Proposition \ref{prop-e_k}, it is easy to see that
$\bigl(\mathbf{W}_k,\mathbf{i}_k\bigr)_{k\in\N}$ is a Markov chain. The state
space $\calZ$ can now be restricted to
$$
\calZ_\Gamma := \biggl\lbrace (xh,n) \, \biggl|\, x \in \bigcup_{i=1}^r R_i, h\in H, n\in\N\biggr\rbrace.
$$
Define $[x_1\dots x_nh]:=x_nh$. Then $([\mathbf{W}_k])_{k\in\N}$ is also a
irreducible Markov chain on a finite state space with invariant probability
measure $\nu$. Thus, we get
$$
\Lambda = \sum_{\substack{xh,yh'\in\bigcup_{i=1}^r \Gamma_i^\times,\\ \tau(xh)\neq \tau(yh')}} \nu(yh')\cdot \frac{\xi\bigl(\tau(xh)\bigr)}{\xi\bigl(\tau(yh')\bigr)}\cdot 
\frac{\partial}{\partial z}\biggl[ \sum_{w\in
\Gamma_{\tau(x)}^\times} p(yh',yh'w)\cdot z \cdot \overline{G}(h'w,xh|z)\biggr]\Biggl|_{z=1}
$$
and
$$
\Delta = \sum_{\substack{xh,yh'\in\bigcup_{i=1}^r \Gamma_i^\times,\\ \tau(xh)\neq
    \tau(yh')}} \nu(yh')\cdot
\frac{\xi\bigl(\tau(xh)\bigr)}{\xi\bigl(\tau(yh')\bigr)}\cdot l(x) 
\cdot \sum_{w\in \Gamma_{\tau(x)}^\times} p(yh',yh'w) \cdot \overline{G}(h'w,xh|1).
$$
Finally, we obtain:
\begin{Cor}\label{exittime-formula}
$$
\lim_{n\to\infty}\frac{l(X_n)}{n}=\frac{\Delta}{\Lambda} \quad \textrm{almost surely.}
$$
\end{Cor}

\subsection{Computation by Double Generating Functions}

\label{dgf}
In this section we derive another formula for the rate of escape with the help
of a theorem of Sawyer and Steger \cite[Theorem 2.2]{sawyer}, which we
reformulate adapted to our situation:
\begin{Th}[Sawyer and Steger]
\label{sawyer}
Suppose we can write for some $\delta>0$
$$
\mathcal{E}(w,z):=\mathbb{E}\biggl( \sum_{n\geq 0} w^{l(X_n)}\,z^{n}\biggl) = \frac{C(w,z)}{g(w,z)}
\quad \textrm{ for } w,z\in (1-\delta;1),
$$
where $C(w,z)$ and $g(w,z)$ are analytic for $|w-1|,|z-1|<\delta$
and \mbox{$C(1,1)\neq 0$}. 
Then
$$
\frac{l(X_n)}{n} \xrightarrow{n\to\infty} \ell=\frac{\frac{\partial}{\partial w}g(1,1)}{\frac{\partial}{\partial z}g(1,1)} \quad
\textrm{ almost surely.}
$$
Moreover, if $(X_n)_{n\in\N_0}$ is a reversible Markov chain, then with $\bar g(r,s):=g(e^{-r},e^{-s})$
$$
\frac{Y_n-n \ell}{\sqrt{n}} \xrightarrow{n\to\infty} N(0,\sigma^2)\
\textrm{ in law, where }\
\sigma^2 = \frac{-\frac{\partial^2}{\partial^2 r}\bar g(0,0)+2\ell
   \frac{\partial^2}{\partial s \partial r} \bar g(0,0) -\ell^2 \frac{\partial^2}{\partial^2 s}\bar g(0,0)}{\frac{\partial}{\partial s}\bar g(0,0)}.
$$
\end{Th}
We remark that \cite[Theorem 2.2]{sawyer} also comprises a central limit theorem.
Similar limit theorems are well-known in analytical
combinatorics, see e.g. Bender and
Richmond \cite{bender-richmond} and Drmota \cite{drmota94}, \cite{drmota97}.
We show now how to write the expectation in the theorem in the required way.
Let $s_H$ be the stopping time of the first return to $H$ after start at
$e$, that is, $s_H=\inf\{1\leq m \in\N \mid
X_m\in H\}$. 
For $h\in H$, $i\in\{1,\dots,r\}$, $x\in\Gamma_i\setminus H$  and $z\in\mathbb{C}$ we define
$$
L(h,x|z)  :=  \sum_{n\geq 0} \Prob_h\bigl[X_n=x,s_H>n\bigr]\,z^n 
=\sum_{y\in \Gamma_i^\times} p(h,y)\cdot z\cdot \overline{G}(y,x|z).
$$
Additionally, we set $L(h,h|z):=1$ and $L(h,h'|z):=0$ for $h'\in H\setminus\{h\}$.
With this notation we have
$$
\mathcal{E}(w,z)  =  \sum_{x\in \Gamma} \sum_{n\in\N_0}
p^{(n)}(e,x)\,z^n\,w^{l(x)} 
= \sum_{x\in \Gamma}\sum_{h\in H} G(e,h|z)\, L(h,x|z)  \,w^{l(x)}.
$$
Setting
\begin{eqnarray*}
\mathcal{L}_i^+(w,z)  & := &  \sum_{x\in \Gamma_i^\times} L(e,x|z)\,w^{l(x)}
\quad \textrm{ and}\\ 
\mathcal{L}_i(w,z)  & := &  \sum_{n\geq 1} \sum_{\substack{x_1\dots x_nh\in \Gamma,\\
    x_1\in \Gamma_i^\times}}  L(e,x_1\dots x_nh|z)\,w^{l(x_1\dots x_nh)},
\end{eqnarray*}
we have
\begin{equation}\label{l-formula}
\mathcal{L}(w,z) := \sum_{x\in \Gamma} L(e,x|z)\,w^{l(x)} = 1 +  \sum_{i=1}^r \mathcal{L}_i(w,z).
\end{equation}
We now rewrite $\mathcal{L}_i(w,z)$:
\begin{eqnarray}
\mathcal{L}_i(w,z)& = &\mathcal{L}_i^+(w,z) \cdot \Bigl( 1 + \sum_{n\geq 2}
\sum_{\substack{x_2\dots x_nh\in\Gamma \setminus H,\\ x_2\notin \Gamma_1}}
L(e,x_2\dots x_nh|z)\,w^{l(x_2\dots x_nh)}\Bigr)\nonumber\\
& = & \mathcal{L}_i^+(w,z) \cdot \Bigl( 1 + \sum_{j=1, j\neq i}^r
\mathcal{L}_j(w,z)\Bigr) 
 =  \mathcal{L}_i^+(w,z) \cdot \Bigl( \mathcal{L}(w,z) -
 \mathcal{L}_i(w,z)\Bigr) \label{li-formula}.
\end{eqnarray}
From (\ref{l-formula}) and (\ref{li-formula}) we obtain
$$
\mathcal{L}(w,z) = 1  + \sum_{i=1}^r \frac{\mathcal{L}_i^+(w,z)\mathcal{L}(w,z) }{1+\mathcal{L}_i^+(w,z)},
$$
yielding
$$
\mathcal{L}(w,z) = \frac{1}{1-\sum_{i=1}^r \frac{\mathcal{L}_i^+(w,z)}{1+\mathcal{L}_i^+(w,z)}}.
$$
Now we can write the expectation of Theorem \ref{sawyer} in the requested way:
\begin{eqnarray*}
\mathcal{E}(w,z)& =& \sum_{h\in H} G(e,h|z) \sum_{x\in\Gamma} L(e,h^{-1}x|z)\,
w^{l(x)}\\
&=& \sum_{h\in H} G(e,h|z) \sum_{x\in\Gamma} L(e,x|z)\,w^{l(x)}
= \frac{\sum_{h\in H} G(e,h|z)}{1-\sum_{i=1}^r \frac{\mathcal{L}_i^+(w,z)}{1+\mathcal{L}_i^+(w,z)}}.
\end{eqnarray*}
Thus, we can apply Theorem \ref{sawyer} with $C(w,z)=\sum_{h\in H} G(e,h|z)$ and
$$
g(w,z) = 1-\sum_{i=1}^r \frac{\mathcal{L}_i^+(w,z)}{1+\mathcal{L}_i^+(w,z)}.
$$
\begin{Cor}\label{dgf-formula}
The rate of escape w.r.t. $l(\cdot)$ is
$$
\lim_{n\to\infty} \frac{l(X_n)}{n}  = \frac{\Upsilon_1}{\Upsilon_2} \textrm{
  almost surely,} 
$$
where
$$
\Upsilon_1 = \sum_{i=1}^r \frac{\sum_{x\in\Gamma_i^\times}
  l(x)\,L(e,x|1)}{\bigl(1+\sum_{x\in\Gamma_i^\times} L(e,x|1)\bigr)^2} \ 
\textrm{ and } \
\Upsilon_2 = \sum_{i=1}^r \frac{\sum_{x\in\Gamma_i^\times}
  L'(e,x|1)}{\bigl(1+\sum_{x\in\Gamma_i^\times} L(e,x|1)\bigr)^2}.
$$
\end{Cor}
\begin{proof}
Computing the derivatives of $g(w,z)$ w.r.t. $w$ and $z$ leads to the proposed formula.
\end{proof}

\subsection{Computation via the Limit Process}

\label{limitprocesses}
In this section we derive another formula for the rate of escape w.r.t. the
natural word length $\Vert \cdot \Vert$. First,
$$
\mathbb{E}[\Vert X_{n}\Vert] = \sum_{\bar g\in \Gamma} \Vert \bar g\Vert\,
\mu^{(n)}(\bar g)\quad \textrm{
  and }\quad
\mathbb{E}[\Vert X_{n+1}\Vert] = \sum_{g,\bar g\in \Gamma} \Vert g\bar g\Vert\,
\mu(g)\, \mu^{(n)}(\bar g).
$$
Thus, we have
$$
\mathbb{E}[\Vert X_{n+1}\Vert] -\mathbb{E}[\Vert X_{n}\Vert]  =   \sum_{g\in
  \Gamma} \mu(g) \int_\Gamma \bigl( \Vert gX_n\Vert
-\Vert X_n\Vert\bigr)\, d\mu^{(n)}.
$$
Since $\mathbb{E}[\Vert X_n\Vert]/n$ converges to $\ell=\lim_{n\to\infty} \Vert
X_n\Vert/n$, it is sufficient to prove that this difference of expectations converges; the limit must
then equal $\ell$.
The process $(X_n)_{n\in\N_0}$ converges to some random element $X_\infty$
valued in
$$
\Gamma_\infty := \biggl\lbrace
x_1x_2\dots \in\Gamma^\N \,\biggl|\, x_i\in \bigcup_{j=1}^r R_j, \tau(x_i)\neq \tau(x_{i+1})
\biggr\rbrace
$$
in the sense that the length of the common prefix of $X_n$ and $X_\infty$ goes to
infinity. We denote by $X_\infty^{(1)}$ the first letter of $X_\infty$ and for
$g\in \bigcup_{i=1}^r \Gamma_i$ we define
$$
Y_{g} := \lim_{n\to\infty} \Vert gX_n\Vert - \Vert X_n \Vert 
= 
\begin{cases}
1,& \textrm{if } X_{\infty}^{(1)}\notin \Gamma_{\tau(g)}\\
-1, & \textrm{if } X_{\infty}^{(1)}\in g^{-1}H \\
0, & \textrm{otherwise}
\end{cases}
$$
At this point we need the equation $\Vert hx\Vert=\Vert x\Vert$ for $h\in H$
and $x\in\Gamma$. This equation is, in general, not satisfied for other length functions.
The Green functions $G(x,y)=\sum_{n\geq 0} \Prob_x[X_n=y]$, where
$x,y\in \Gamma_i$ for any $i\in\{1,\dots,r\}$, satisfy the following linear recursive 
equations:
\begin{displaymath}
G(x,y)  =  \delta_{x}(y) + \sum_{w\in \Gamma_i} p(x,w)\, G(w,y) +
\sum_{\substack{xwh\in\Gamma,\\ \Vert xwh\Vert=2}} p(x,xwh) \sum_{h'\in H} H(wh,h'|1)\,G(xh',y).
\end{displaymath}
This system of Green functions can be solved, when the functions $H(wh,h'|1)$
can be obtained by solving (\ref{h-equations}). We now define
$$
\varrho(i) := \Prob\bigl[X_\infty^{(1)}\in \Gamma_i\bigr] =
\sum_{h\in H} G(e,h|1) \sum_{g\in\Gamma_i^\times} \mu(g) \cdot \Bigl(1-
\sum_{h'\in H} H(hg,h'|1)\Bigr).
$$
By transience, $\sum_{i=1}^r \varrho(i)=1$. Furthermore, 
$\Prob[Y_{g}=1] =  1-\varrho\bigl( \tau(g)\bigr)$ and
$$
\Prob[Y_{g}=-1]  =  \sum_{h\in H} F(e,g^{-1}h) \cdot \bigl(1-\varrho\bigl(
\tau(g)\bigr)\bigr)
= \frac{1-\varrho\bigl(\tau(g)\bigr)}{G(e,e)} \sum_{h\in H} G(e,g^{-1}h).
$$
By Lebesgue's Dominated Convergence Theorem,
$$
\E\bigl[\Vert X_{n+1}\Vert \bigr] - \E\bigl[\Vert X_{n}\Vert\bigr]
\xrightarrow{n\to\infty} \sum_{i=1}^r \sum_{g\in\Gamma_i^\times} \mu(g)
\Bigl( \Prob[Y_{g}=1]- \Prob[Y_{g}=-1]\Bigr).
$$
But this limit must be the rate of escape $\ell$. Thus:
\begin{Cor}\label{lp-formula}
$$
\ell =  
\lim_{n\to\infty} \frac{\Vert X_n\Vert}{n}=
\sum_{i=1}^r \biggl[\mu(\Gamma_i^\times) \bigl(1-\varrho(i)\bigr)
- \frac{1-\varrho(i)}{G(e,e)} \sum_{g\in\Gamma_i^\times}
\sum_{h\in H} \mu(g) G(e,g^{-1}h)\biggr].
$$
\end{Cor}
As a final remark observe that the formulas of Corollaries \ref{exittime-formula},
\ref{dgf-formula} and \ref{lp-formula} have complexities in decreasing order:
while the computation of the rate of escape by Corollary \ref{exittime-formula}
needs three systems of equations to be solved and derivatives to be calculated,
the computation by Corollaries \ref{dgf-formula} or \ref{lp-formula} needs
only two systems of equations to be solved, while the formula in Corollary
\ref{dgf-formula} deals also with derivatives.

\section{Sample Computations}
\label{sample}

\subsection{A Regular Language}
Let be $A=\{a,b,c\}$ and we set $l(a)=l(b)=l(c)=1$. We consider the set $\mathcal{L}$ of all words over the
alphabet $A$, such that in each $w\in\mathcal{L}$ the letter $b$ is the first
letter of $w$ or follows after the letter $a$ and the letter $c$ may only appear after the letter
$b$; e.g., $abcaba \in \mathcal{L}$, but $abcba \notin \mathcal{L}$. Consider the random walk on $\mathcal{L}$ given by the following
transition probabilities:
\begin{eqnarray*}
&&p(aa,aaa)=\frac{1}{3},\ p(aa,aab)=\frac{1}{3},\ p(aa,a)=\frac{1}{3},\ p(ab,aba)=\frac{1}{6},\ p(ab,abc)=\frac{1}{3},\ p(ab,a)=\frac{1}{2},\\
&&p(ba,baa)=\frac{1}{4},\ p(ba,bca)=\frac{1}{4},\ p(ba,bab)=\frac{1}{4},\ p(ba,a)=\frac{1}{4},\\
&&p(bc,bca)=\frac{1}{2},\ p(bc,a)=\frac{1}{2},\ \
p(ca,caa)=\frac{1}{4},\ p(ca,cab)=\frac{1}{2},\ p(ca,a)=\frac{1}{4}.
\end{eqnarray*}
Note that it is not necessary to specify any further transition probabilities,
as the formula for the rate of escape does not depend on the transition probabilities of the form 
\mbox{$\Prob[X_{n+1}=w' \mid X_n=w]$,} where $w\in \{\varepsilon,a,b,c\}$. The
system of equations \ref{h-equations} is then
\begin{eqnarray*}
H(aa,a|z) & = & \frac{z}{3}  \bigl( H(aa,a|z)\cdot H(aa,a|z) + H(ab,a|z)\cdot
H(aa,a|z)+1\bigr),\\
H(ab,a|z) & = & \frac{z}{3}  H(bc,a|z)\cdot H(aa,a|z) + \frac{z}{6} H(ba,a|z)\cdot
H(aa,a|z) + \frac{z}{2},\\
H(ba,a|z) & = & \frac{z}{4}  \bigl( H(aa,a|z)\cdot H(ba,a|z) + H(ca,a|z)\cdot
H(ba,a|z)  
+ H(ab,a|z)\cdot H(ba,a|z) +1 \bigr),\\
H(bc,a|z) & = & \frac{z}{2} H(ca,a|z)\cdot H(ba,a|z) + \frac{z}{2},\\
H(ca,a|z) & = & \frac{z}{4} H(aa,a|z)\cdot H(ca,a|z) + \frac{z}{2}
H(ab,a|z)\cdot H(ca,a|z) + \frac{z}{4}.
\end{eqnarray*}
This system in the unknown variables $H(aa,a|z)$, $H(ab,a|z)$, $H(ba,a|z)$,
$H(bc,a|z)$ and $H(ca,a|z)$, where $z$ appears as a parameter, can be solved
with the help of \textsc{Mathematica}. With these solutions we can compute the
modified Green functions $\overline{G}(\cdot,\cdot|z)$ by solving the linear
system (\ref{g-system}). Note that only $\overline{G}(aa,aa|z)$, $\overline{G}(ab,aa|z)$, $\overline{G}(ba,ba|z)$, $\overline{G}(bc,ba|z)$, $\overline{G}(ca,ca|z)$ are non-zero functions.
Moreover, we get
\begin{eqnarray*}
\xi(aaa) & = & \xi(baa)=\xi(caa) =\frac{1}{3} \bigl( 1- H(aa,a|1)\bigr) + \frac{1}{3} \bigl( 1-
H(ab,a|1)\bigr),\\
\xi(aab) & = & \xi(bab)=\xi(cab)=\frac{1}{6} \bigl( 1- H(ba,a|1)\bigr) + \frac{1}{3} \bigl( 1-
H(bc,a|1)\bigr),\\
\xi(aba) & = & \frac{1}{4} \bigl( 1- H(aa,a|1)\bigr) + \frac{1}{4} \bigl( 1-
H(ca,a|1)\bigr) + \frac{1}{4} \bigl(1- H(ab,a)\bigr),\\
\xi(abc) & = & \frac{1}{2} \bigl( 1- H(ca,a|1)\bigr),\
\xi(bca)  =  \frac{1}{4} \bigl( 1- H(aa,a|1)\bigr) + \frac{1}{2} \bigl( 1-
H(ab,a|1)\bigr).
\end{eqnarray*}
Since $\nu(abc)=\sum_{def\in \overline{A}^3} \nu(def)\ \tilde q(def,abc)$, we can compute
the invariant measure as
$$
\begin{array}{c}
\nu(aaa)=0.32475,\, \nu(aab)=0.13194,\, \nu(aba)=0.12597, \,
\nu(abc)=0.08021, \, \nu(baa)=0.05350,\\[1ex]
\nu(bca)=0.13095, \,
\nu(bab)=0.02174, \, \nu(caa)=0.07844, \, \nu(cab)=0.05251.
\end{array}
$$
Now we have all necessary ingredients to compute
$\Lambda =3.78507$, and finally we get the rate of escape as $\ell=0.264196$.

\subsection{$\Z/d\Z \ast_{\Z/2\Z} \Z/d\Z$}

Consider the free product by amalgamation $\Z/d\Z \ast_{\Z/2\Z} \Z/d\Z$,
$d\in\N$ even, over the common subgroup $\Z/2\Z$. Suppose that $\Z/d\Z$ is
generated by some element $a$ with $a^6$ equal to the identity. Setting
$\mu_1(a)=\mu_2(a)=1$ and $\alpha_1=\alpha_2=1/2$ we get the following values
for the rate of escape $\ell$ w.r.t. $\|\cdot \|$:
$$
\begin{array}{c|c|c|c|c}
d & 6 & 8 & 10 & 12\\
\hline
\ell & 0.24749 & 0.40859 & 0.46144 & 0.47543
\end{array}
$$


\bibliographystyle{abbrv}
\bibliography{literatur}

\end{document}